\documentclass[a4paper,reqno,12pt]{amsart}
\usepackage[footskip=0.5in, headheight = 0.5in, top=1.25in, bottom=1.25in, right=1.2in, left=1.2in]{geometry}

\usepackage{macros}

\title[Counting equitable $k$-colorings]{Counting equitable $k$-colorings in graphs of bounded clique-width}
\author{Holger Dell$^{\dagger}$}
\author{Thore Husfeldt$^{\dagger}$}
\author{Amir Nikabadi$^{\dagger,\ast}$}
\date{\today}

\address{$^\dagger$IT University of Copenhagen, Denmark.}
\address{$^\ast$Supported by the Independent Research Fund Denmark (DFF) grant agreement number 2098-00012B.}

\allowdisplaybreaks
\begin{document}

\begin{abstract}
For a graph $G$, a proper $k$-coloring of $G$ is \emph{equitable} if the sizes of any two color classes differ by at most one. The \textsc{Equitable $k$-Coloring} problem asks, for a given graph $G$ and integer $k$, whether $G$ admits an equitable $k$-coloring.
Bodlaender and Fomin (Theoretical Computer Science 2005) showed that it is polynomial-time solvable on graphs of bounded treewidth, while it remains $\NP$-hard on cographs, and thus on graphs of constant clique-width.
Fellows et al. (Information and Computation 2011) showed that the problem becomes $\mathsf{W[1]}$-hard when parameterized by tree-width (and hence clique-width) plus the number of colors~$k$.

We first show that, there exists an algorithm, given an integer $k\ge 1$ and an $n$-vertex graph $G$ together with a $w$-expression whose underlying unlabelled graph is $G$, computes the number of equitable $k$-colorings of $G$ in time
$2^{O(k\cdot w)}\cdot n^{O(k)}$. In particular, we show that for every fixed $k$, counting equitable $k$-colorings is polynomial-time solvable on graph classes of bounded clique-width, given a clique-width expression.

We then show that, under $\mathsf{SETH}$, the dependence on clique-width in this algorithm is essentially optimal. As a consequence, our results provide a fairly tight picture of the complexity of \textsc{Equitable $k$-Coloring} with respect to the combined parameter $k$+clique-width in the following sense: For variable $k$, the problem is $\mathsf{W[1]}$-hard, however for every fixed integer $k$, it is polynomial-time solvable on graphs of bounded clique-width given a clique-width expression, and this remains true even for the counting version.

Second, we refine our clique-width algorithm for the linear setting. We show that there exists an algorithm, given an integer $k\ge 1$ and an $n$-vertex graph $G$ together with a linear $w$-expression constructing $G$, computes the number of equitable $k$-colorings of $G$ in time $\max\{1,2^k-2\}^w\cdot n^{k+O(1)}$.
Thus, for bounded linear clique-width, we obtain a significantly sharper dependence on the width parameter than in the general clique-width case.

Third, we consider a different structural restriction, namely the class of $P_t$-free graphs.
A graph is called $P_t$-free if it does not contain the path on $t$ vertices as an induced subgraph.
This is a different setting from bounded clique-width; in particular, already $P_5$-free graphs have unbounded clique-width.
Nevertheless, we show that for every $P_t$-free graph $G$, the number of equitable list $3$-colorings of $G$ can be computed in subexponential time.
\end{abstract}

\clearpage

\maketitle

\clearpage

\section{Introduction}
Let $G$ be a graph. A proper $k$-coloring of $G$ is \emph{equitable} if the sizes of any two color classes differ by at most one. Equitable coloring has a long history.
A celebrated result of Hajnal and Szemerédi~\cite{hajnal1970proof} settled a conjecture of Erdős by showing that every graph $G$ with maximum degree at most $k$ admits an equitable $(k+1)$-coloring.
In contrast to the classical setting, where Brooks' theorem states that every connected graph $G$ with maximum degree $\Delta$ satisfies $\chi(G)\le \Delta$, except when $G$ is a complete graph or an odd cycle; the equitable analogue of this statement is still open.
In particular, Chen, Lih, and Wu~\cite{chen1994equitable} conjectured that every connected graph $G$ with maximum degree $\Delta \geq 2$ admits an equitable coloring with $\Delta$ colors, except when $G$ is a complete graph, an odd cycle, or $\Delta$ is odd and $G=K_{\Delta,\Delta}$.
This conjecture remains open in general, although it has been verified for several graph classes~\cite{chen1994equitable, chen1994-treeequitable, chen2009equitable, kostochka2005equitable, zhang1998equitable, de1985some}.

\problemStatement{\textsc{Equitable $k$-Coloring}}{
	Input={A graph $G$ and an integer $k$.},
	Task={Decide whether $G$ admits an equitable $k$-coloring.}
}

\smallskip

On the algorithmic side, the \textsc{Equitable $k$-Coloring} problem has been studied extensively.
By a simple reduction from \textsc{$k$-Coloring}, it follows that \textsc{Equitable $k$-Coloring} is $\NP$-hard.
Bodlaender and Fomin~\cite{bodlaender2005equitable} showed that the problem is solvable in polynomial time on graphs of bounded treewidth, while it remains hard on cographs.
Fellows et al.~\cite{fellows2011complexity} showed that the problem is $\mathsf{W[1]}$-hard when parameterized by treewidth together with the number of colors.
The parameterized complexity of \textsc{Equitable $k$-Coloring} has also been studied under structural parameterizations~\cite{gomes2020structural} and on restricted graph classes~\cite{gomes2018parameterized, cordasco2020iterated}.

In this paper, we first study \textsc{Equitable $k$-Coloring} with respect to the combined parameters clique-width and the number of colors.
The \emph{clique-width} (introduced by Courcelle et al.~\cite{courcelle1993handle}) of a graph $G$, denoted by $\cw(G)$, is the minimum
number of different labels needed to construct $G$ using certain graph operations (see~\Cref{sec:prelim}
for details). Clique-width has been studied extensively both in algorithmic and structural
graph theory. It generalizes the classical notion of \emph{treewidth}, denoted by $\tw$, in the following sense: For every graph $G$, $\cw(G)\leq O(2^{\tw(G)})$~\cite{corneil2005relationship}, whereas no converse bound holds, since there are graphs of arbitrarily large treewidth and bounded clique-width.

The importance of clique-width stems in part from the fact that many problems, including 
\textsc{Coloring} and \textsc{Hamilton Cycle}~\cite{espelage2001solve, kobler2003edge, courcelle2000linear, rao2007msol},
that are $\NP$-hard in general become polynomial-time solvable on any
graph class, say $\mathcal{G}$, of bounded clique-width, that is, for which there exists a constant $c$, such that for every graph $ G \in \mathcal{G}$, $\cw(G)\leq c$.
It follows from the result of~\cite{fellows2011complexity} that \textsc{Equitable $k$-Coloring} remains $\mathsf{W[1]}$-hard when parameterized by the clique-width of the input graph together with the number of colors.
We show, however, that counting equitable $k$-colorings can still be done in $\mathsf{XP}$ for these combined parameters.
In particular, we prove the following.

\begin{theorem}\label{tmh:cw-restate}
There exists an algorithm that, given an integer $k\ge 1$ and an $n$-vertex graph $G$ together with a $w$-expression whose underlying unlabelled graph is $G$, computes the number of equitable $k$-colorings of $G$ in time
$2^{O(k\cdot w)}\cdot n^{O(k)}$. In particular, for every fixed $k$, counting equitable $k$-colorings is polynomial-time solvable on graph classes of bounded clique-width, given a clique-width expression.
\end{theorem}

The proof of~\Cref{tmh:cw-restate} is by dynamic programming over a given clique-width expression.
Let us distinguish our approach from more standard dynamic programming algorithms over clique-width expressions.
While the overall framework is the usual bottom-up dynamic programming over a given expression, the state representation is different. Ordinary coloring needs only local information about how colors interact with the current label classes.
For equitable coloring, however, one must also keep track of the global sizes of the color classes.
Accordingly, for each subexpression we use a signature that records, for every label, the set of colors appearing on that label class, and, for every color, the current size of its color class.
It is this additional global information that makes it possible to enforce the equitable condition at the root.

We then present a complementary $\mathsf{SETH}$-based lower bound for graphs of bounded clique-width.
Although it does not fully match the running time in~\Cref{tmh:cw-restate}, it shows that for every fixed $k\ge 3$ one cannot hope for a running time of the form $O^\ast((2^k-2-\varepsilon)^{\cw(G)})$.
More precisely, we prove the following.

\begin{theorem}\label{thm:cw-seth-restate}
Let $G$ be a graph.
For every fixed $k\ge 3$ and every $\varepsilon>0$, unless SETH fails, \textsc{Equitable $k$-Coloring} cannot be solved in time
$O^\ast\bigl((2^k-2-\varepsilon)^{\cw(G)}\bigr)$,
where $\cw(G)$ denotes the clique-width of $G$.
\end{theorem}

The proof of~\Cref{thm:cw-seth-restate} is based on the $\mathsf{SETH}$-based lower bound of Lampis~\cite{lampis2020finer} for \textsc{$k$-Coloring} parameterized by clique-width.
We reduce from \textsc{$k$-Coloring} by padding the input graph with an independent set, in such a way that proper $k$-colorings of the original graph correspond exactly to equitable $k$-colorings of the new graph, while the clique-width increases by at most one.

We next refine our~\Cref{tmh:cw-restate} by turning to linear clique-width, which plays a role analogous to that of pathwidth relative to treewidth.
Because linear expressions have a prescribed left-to-right structure (we will get back to this in~\Cref{sec:lcw}), they admit a more economical dynamic-programming state space.
This yields a sharper running time than in the general clique-width setting.

\begin{theorem}\label{thm:lcw-counting-restate}
There exists an algorithm that, given an integer $k\ge 1$ and an $n$-vertex graph $G$ together with a linear $w$-expression constructing $G$, computes the number of equitable $k$-colorings of $G$ in time $\max\{1,2^k-2\}^w\cdot n^{k+O(1)}$.
In particular, for every fixed integer $k$, counting equitable $k$-colorings is polynomial-time solvable on graph classes of bounded linear clique-width, given a linear clique-width expression.
\end{theorem}

We finally turn to $P_t$-free graphs. Since graphs excluding a fixed induced subgraph form one of the most natural classes in structural graph theory, $P_t$-free graphs provide a particularly appealing test case: they are far beyond bounded treewidth, and even already $P_5$-free graphs have unbounded clique-width~\cite{dabrowski2016clique}. 

For ordinary \textsc{$k$-Coloring} (which asks whether a given graph is $k$-colorable) on $P_t$-free graphs, polynomial-time algorithms are known for $t\leq 5$~\cite{hoang2010deciding}, for $(k,t)=(4,6)$~\cite{chudnovsky2019four}, and for $(k,t)=(3,7)$~\cite{bonomo2018three}, while Huang~\cite{huang2016improved} showed that \textsc{$4$-Coloring} is $\NP$-complete for $P_7$-free graphs and \textsc{$5$-Coloring} is $\NP$-complete for $P_6$-free graphs; in particular, the complexity of \textsc{$3$-Coloring} on $P_t$-free graphs remains open for $t\geq 8$.
This makes $P_t$-free graphs a particularly natural setting in which to seek subexponential-time algorithms for equitable coloring.
Groenland et al.~\cite{groenland2019h} showed that the number of proper $3$-colorings of a $P_t$-free graph can be computed in subexponential time.
We extend this result to equitable colorings.
In particular, we prove the following.

\begin{theorem}\label{thm:PT-restate}
Let $t\in\mathbb{N}$.
There is an algorithm that, given a $P_t$-free graph $G$ on $n$ vertices and a list assignment
$L\colon V(G)\to 2^{\{1,2,3\}}$, computes the number of equitable proper list $3$-colorings of $(G,L)$ in time
$2^{O(\sqrt{n\log n})}\cdot \mathrm{poly}(n)$.
\end{theorem}

The remainder of the paper is organized as follows.
In~\Cref{sec:cw&seth} we prove~\Cref{tmh:cw-restate} and~\Cref{thm:cw-seth-restate}.
In~\Cref{sec:lcw} we prove~\Cref{thm:lcw-counting-restate}.
Finally, in~\Cref{sec:pt} we prove~\Cref{thm:PT-restate}.

\section{Preliminaries}\label{sec:prelim}

\paragraph*{Sets and numbers} We use $\mathbb{N}$ to denote the set of positive integers. We let $[n] = \{1,\dots, n\}$ for every $n\in \mathbb{N}$.
For a set $U$, we denote by $2^{U}$ the set of all subsets of $U$. 
For $S \subseteq U$, we let $\overline{S} = U \setminus S$. For $k \in \mathbb{N}$, by
a \textit{$k$-subset} of $U$ we mean a subset of $U$ with $k$ elements. The set of all $k$-subsets of $U$ is denoted by $U^{(k)}$. The identity map on $U$ is denoted by  $\mathrm{id}_{U}$.
We write $O^\ast(f(n))$ for a non-decreasing function $f(n)$ of the input size parameter $n$ to suppress polynomial factors in $n$. For two sets $A$ and $B$ we denote by $A \uplus B$ the disjoint union of $A$ and $B$. Let $f:D\to R$ be a function and let $A\subseteq D$.
The \emph{restriction of $f$ to $A$}, denoted by $f|_{A}$, is the function $f|_{A}:A\to R$ defined by
$f|_{A}(x)\coloneqq f(x)$ for all $x\in A$.
For composable maps $f$ and $g$, $g\circ f$ denotes their composition.

\paragraph*{Graphs}
Graphs in this paper have finite vertex sets and no loops or parallel edges. Let $G = (V,E)$ be a graph. A \textit{clique} in $G$ is a set of pairwise adjacent vertices. A \textit{stable set} or \textit{independent set} in $G$ is a set of pairwise non-adjacent vertices.
For $X \subseteq V(G)$, we denote the subgraph of $G$ induced by $X$ as $G[X]$, that is 
$G[X] = (X, \{uv \colon u, v \in X \mbox{ and } uv \in E \})$.

We denote by $N(X)$ vertices of $G\setminus X$ with a neighbor in $X$, and $N[X] = N(X)\cup X$. For disjoint $X, Y \subseteq V(G)$, we say that $X$ is \textit{complete} to $Y$ if every vertex in $X$ is adjacent to every vertex in $Y$, and $X$ is \textit{anticomplete} to $Y$ if there are no edges between $X$ and $Y$.

For given graphs $G$ and $H$, we say that $G$ is \emph{$H$-free} if $G$ does not contain $H$ as an induced subgraph. We let $P_n$, $C_n$, and $K_n$ denote the chordless path, chordless cycle, and the complete graph on $n$ vertices.

A \textit{coloring} of $G$ is a mapping $\varphi \colon V\rightarrow \{1, 2, \ldots\}$ that gives each vertex $u \in V$ a \textit{color} $\varphi(u)$ in such a way that, for every two adjacent vertices $u$ and $v$ in $G$, we have that $\varphi(u) \neq \varphi(v)$. For $k \geq 1$, a coloring $\varphi$ of $G$ is a \textit{$k$-coloring} if $\varphi(u) \in \{1, \ldots, k\}$ for every $u \in V$, and a graph is \textit{$k$-colorable} if it admits a $k$-coloring.
The \textit{chromatic number}~$\chi(G)$ is the smallest $k$ for which $G$ is $k$-colorable.
A $k$-coloring of $G$ is \textit{equitable} if $||\varphi^{-1}(i)|-|\varphi^{-1}(j)||\le 1$ for every $i,j\in [k]$. Equivalently, $G$ is equitably $k$-colorable if it is $k$-colorable and every color is used either $\lfloor n/k \rfloor$ or $\lceil n/k \rceil$ times where $n=|V(G)|$. The \textit{equitable chromatic number}~$\eqcol{G}$ is the smallest $k$ for which $G$ is equitably $k$-colorable.

\paragraph*{Clique-width and linear clique-width}
For a fixed $w\ge 1$, a $w$-labelled graph is a graph $G$ together with a labelling function $\lab:V(G)\to [w]$.
Equivalently, it may be described by the partition
$V(G)=C_1\uplus\cdots\uplus C_w,$
where $C_i=\lab^{-1}(i)$ is the set of vertices of label $i$.
On $w$-labelled graphs we consider the following operations:
\begin{itemize}
    \item for each $i\in [w]$, the introduce-operation $i(v)$, which constructs a single-vertex graph whose unique vertex $v$ has label $i$;
    \item the union-operation $\union(\cdot,\cdot)$, which forms the disjoint union of two $w$-labelled graphs;
    \item for distinct $i,j\in [w]$, the relabel-operation $\ren_{i\to j}(\cdot)$, which changes the label of every vertex of label $i$ to $j$;
    \item for distinct $i,j\in [w]$, the join-operation $\join_{i,j}(\cdot)$, which adds all edges with one endpoint of label $i$ and the other of label $j$.
\end{itemize}

A \emph{$w$-expression} is any valid term built from these operations.
If $\mu$ is a $w$-expression, we write $G_\mu$ for the $w$-labelled graph constructed by $\mu$, and $\lab_\mu$ for its
labelling function.
The \emph{clique-width} of an unlabelled graph $G$, denoted by $\cw(G)$, is the least integer $w$ such that there exists a
$w$-expression whose underlying unlabelled graph is $G$.
Associated with every $w$-expression $\mu$ is its syntax tree $T_\mu$ in the natural way.
For any node $t\in V(T_\mu)$, the subtree rooted at $t$ induces a subexpression $\mu_t$.
Whenever $\mu$ is fixed, we write
$G_t\coloneqq G_{\mu_t},
V_t\coloneqq V(G_t),
E_t\coloneqq E(G_t),
\lab_t\coloneqq \lab_{\mu_t}.$
We use this notation to refer conveniently to the labelled graph associated with a node of the syntax tree of $\mu$.

A clique-expression $\mu$ is \emph{linear} if in every union-operation, the second operand is a single-vertex
$w$-labelled graph.
The \emph{linear clique-width} of an unlabelled graph $G$, denoted by $\lcw(G)$, is the least integer $w$ such that
$G$ is the underlying unlabelled graph of some linear $w$-expression.

\paragraph*{Complexity} For basics of parameterized complexity, see~\cite{cygan2015parameterized}.

The \emph{Strong Exponential-Time Hypothesis} ($\mathsf{SETH}$)~\cite{impagliazzo2001problems, calabro2009complexity} is a standard complexity assumption that is often used to rule out algorithms with substantially improved exponential dependence. It concerns the complexity of \textsc{$k$-Satisfiability}, where every clause contains at most $k$ literals. More precisely, for every $k\ge 3$, let
\[
c_k\coloneqq \inf \bigl\{ \delta : k\text{-\textsc{Satisfiability} can be solved in time } O(2^{\delta n}) \bigr\},
\]
where $n$ denotes the number of variables in the input formula.
Then, the Strong Exponential-Time Hypothesis states that
$\lim_{k\to\infty} c_k = 1.$
\section{Graphs of bounded Clique-width: The Proofs of~\Cref{tmh:cw-restate} and~\Cref{thm:cw-seth-restate}}\label{sec:cw&seth}

Bodlaender and Fomin~\cite{bodlaender2005equitable} showed that \textsc{Equitable $k$-Coloring} can be solved in
polynomial time on graphs of bounded treewidth. In this section we show that, for every fixed integer $k$, given an $n$-vertex graph $G$ together with a $w$-expression 
whose underlying unlabelled graph is $G$, one can compute the number of equitable $k$-colorings of $G$ in time
$2^{O(k\cdot w)}\cdot n^{O(k)}$.
In particular, this yields a polynomial-time algorithm on graph classes of bounded clique-width.

We will use the clique-width operations introduced in~\Cref{sec:prelim}. In particular, a $w$-expression is a term
whose leaves are single-vertex $w$-labelled graphs and whose internal nodes apply one of the operations
$\union(\cdot,\cdot)$, $\join_{i,j}(\cdot)$ ($i\neq j$), and $\ren_{i\to j}(\cdot)$ ($i\neq j$). We first introduce the signature notation used by the dynamic programming algorithm.

\medskip

\begin{definition}[$X$-signature]\label{def:X-signature}
Let $k,w\in\mathbb{N}$.
Let $X$ be a subexpression of a given $w$-expression, and let $G_X$ be the $w$-labelled graph constructed by $X$, with vertex set $V_X$ and labelling function
$\lab_X:V_X\to[w]$.
An \emph{$X$-signature} is a pair $\sig=(\mathbf{S},\mathbf{a})$ where $\mathbf{S}=(S_1,\dots,S_w)$ with
$S_\ell\subseteq [k]$ for all $\ell\in[w]$, and $\mathbf{a}=(a_1,\dots,a_k)\in\{0,1,\dots,|V_X|\}^k$.
\end{definition}

We say that a proper $k$-coloring $\varphi:V_X\to[k]$ \emph{is represented by} $\sig=(\mathbf{S},\mathbf{a})$ if
\begin{itemize}
  \item $S_\ell=\{\,c\in[k]:\exists v\in V_X \text{ with }\lab_X(v)=\ell \text{ and }\varphi(v)=c\,\}$ for all
        $\ell\in[w]$, and
  \item $a_c=|\varphi^{-1}(c)|$ for all $c\in[k]$.
\end{itemize}

The key point is that this signature stores precisely the information needed by the clique-width operations:
for join-operations it is enough to know which colors appear on each label class, while for equitable coloring
one must additionally know the global sizes of the color classes.

\begin{lemma}\label{lem:root-condition-cw}
 Let $n=kq+r$ with
$q=\lfloor n/k\rfloor$ and $0\le r<k$.
If $X^\star$ is a $w$-expression for $G$, then $G$ admits an equitable $k$-coloring if and only if there exists an $X^\star$-signature
$\sig=(\mathbf{S},\mathbf{a})$ represented by a proper $k$-coloring of $G$ such that
$a_c\in\{q,q+1\}$ for every $c\in[k]$, and $\left|\{c\in[k]:a_c=q+1\}\right|=r$.
\end{lemma}

\begin{proof}
Let $\varphi$ be a proper $k$-coloring of $G$, and let
$\mathbf a=(a_1,\dots,a_k)$ be its color-class size vector, where
$a_c\coloneqq |\varphi^{-1}(c)|$ for each $c\in[k]$.
Then $\varphi$ is equitable if and only if all entries of $\mathbf a$ differ by at most $1$. Equivalently, if and only if, with
$q'\coloneqq \min_{c\in[k]} a_c,$
we have $a_c\in\{q',q'+1\}$ for all $c\in[k]$.
In that case,
\[
n=\sum_{c=1}^k a_c
  = kq' + \left|\{c\in[k]:a_c=q'+1\}\right|.
\]
Since $q'=\min_{c\in[k]} a_c$, there exists at least one color $c$ with $a_c=q'$, and therefore
$0\le \left|\{c\in[k]:a_c=q'+1\}\right|<k.$
This is the Euclidean division of $n$ by $k$.
By uniqueness of Euclidean division, $q'=q$ and
$\left|\{c\in[k]:a_c=q+1\}\right|=r.$
Therefore, $G$ admits an equitable $k$-coloring if and only if there exists a proper $k$-coloring of $G$
whose size vector $\mathbf a$ satisfies $a_c\in\{q,q+1\}$ for all $c\in[k]$, and exactly $r$ coordinates are equal
to $q+1$.
Since an $X^\star$-signature represented by $\varphi$ records exactly the vector $\mathbf a$, the claim follows.
\end{proof}

\begin{definition}\label{def:sig-map-relabel}
Let $X=\ren_{i\to j}(Y)$ with $i\neq j$. For a $Y$-signature $\sig=(\mathbf{S},\mathbf{a})$ and an $X$-signature
$\sig'=(\mathbf{S}',\mathbf{a}')$, we write $\sig\mapsto_{i\to j}\sig'$ if $\mathbf{a}'=\mathbf{a}$ and
\[
S'_\ell=
\begin{cases}
S_\ell & \text{if }\ell\notin\{i,j\},\\
\emptyset & \text{if }\ell=i,\\
S_i\cup S_j & \text{if }\ell=j.
\end{cases}
\]
\end{definition}

We can now prove the main result of this section.

\begin{theorem}\label{thm:cw-counting}
Fix an integer $k\ge 1$.
There exists an algorithm that, given an $n$-vertex graph $G$ together with a $w$-expression $\mathcal{E}$ of size
$m$ whose underlying unlabelled graph is $G$, computes the number of equitable $k$-colorings of $G$ in time
$2^{O(k\cdot w)}\cdot m\cdot n^{O(k)}$
(up to polynomial factors for arithmetic on integers of size at most $k^n$).
In particular, for standard polynomial-size encodings of clique-width expressions, the running time is
$2^{O(k\cdot w)}\cdot n^{O(k)}$.
\end{theorem}

\begin{proof}
Let $\mathcal{E}$ be the given $w$-expression, and let $X^\star\coloneqq \mathcal{E}$.
For every subexpression $X$ of $\mathcal{E}$, we denote by $G_X$ the $w$-labelled graph constructed by $X$,
with vertex set $V_X$ and labelling function $\lab_X:V_X\to[w]$.
For an $X$-signature $\sig=(\mathbf{S},\mathbf{a})$, we write $\mathbf{S}=(S_1,\dots,S_w)$ and
$\mathbf{a}=(a_1,\dots,a_k)$.
Recall that a proper $k$-coloring $\varphi:V_X\to[k]$ is \emph{represented by} $\sig$ if
\begin{itemize}
  \item $S_\ell=\{\,c\in[k]:\exists v\in V_X \text{ with }\lab_X(v)=\ell \text{ and }\varphi(v)=c\,\}$ for all
        $\ell\in[w]$, and
  \item $a_c=|\varphi^{-1}(c)|$ for all $c\in[k]$.
\end{itemize}

For every subexpression $X$ and every $X$-signature $\sig$, define
\[
\mathrm{count}[X,\sig]\in\mathbb{Z}_{\ge 0}
\]
to be the number of proper $k$-colorings $\varphi$ of $G_X$ that are represented by $\sig$.
Equivalently,
\[
\mathrm{count}[X,\sig]
=
\bigl|\{\varphi : \varphi \text{ is a proper $k$-coloring of }G_X \text{ and is represented by }\sig\}\bigr|.
\]
We compute all values $\mathrm{count}[X,\sig]$ by dynamic programming over the subexpressions $X$ of $\mathcal{E}$,
proceeding from smaller subexpressions to larger ones.
We show, by induction on the structure of $X$, that for every subexpression $X$ of $\mathcal{E}$ and every
$X$-signature $\sig$, the value $\mathrm{count}[X,\sig]$ equals the number of proper $k$-colorings of $G_X$
represented by $\sig$.
The base case is when $X$ is a leaf.
For the induction step, we assume the claim holds for the immediate proper subexpressions of $X$, and prove it for $X$ according to whether the last operation of $X$ is a union, join, or relabel operation.

\medskip
\noindent\textbf{Leaf.}
Suppose that $X$ is a leaf subexpression creating a single vertex $v$ with label $\ell\in[w]$.
For $c\in[k]$, let $\mathbf e_c\in\{0,1\}^k$ be the vector whose $c$th coordinate is $1$ and whose other coordinates are $0$.
For each color $c\in[k]$, define the signature $\sig^{(\ell,c)}=(\mathbf{S},\mathbf{a})$ by
$S_\ell=\{c\}$, $S_{\ell'}=\emptyset$ for all $\ell'\in[w]\setminus\{\ell\}$, and $\mathbf a=\mathbf e_c$.
We set
\[
\mathrm{count}[X,\sig]
\coloneqq 
\begin{cases}
1 & \text{if }\sig=\sig^{(\ell,c)}\text{ for some }c\in[k],\\
0 & \text{otherwise.}
\end{cases}
\]
This is correct because $G_X$ has a single vertex and no edges, so for each $c\in[k]$ there is exactly one proper
coloring with $\varphi(v)=c$, and it represents exactly $\sig^{(\ell,c)}$; no other signature can be represented.

\medskip
\noindent\textbf{Union.}
Assume $X=\union(Y,Z)$. By the semantics of $\union$, the graphs $G_Y$ and $G_Z$ are vertex-disjoint induced
subgraphs of $G_X$, there are no edges between $V_Y$ and $V_Z$, and
\begin{equation}\label{eq:union-structure}
V_X=V_Y\uplus V_Z, \quad E(G_X)=E(G_Y)\uplus E(G_Z), \quad
\lab_X|_{V_Y}=\lab_Y, \lab_X|_{V_Z}=\lab_Z.
\end{equation}

Fix an $X$-signature $\sig=(\mathbf{S},\mathbf{a})$.
We define $\mathrm{count}[X,\sig]$ by
\begin{equation}\label{eq:count-union-rigorous}
\mathrm{count}[X,\sig]
\coloneqq 
\sum_{\mathbf{a}^Y+\mathbf{a}^Z=\mathbf{a}}
\ \sum_{\substack{\mathbf{S}^Y,\mathbf{S}^Z:\\
S_\ell=S_\ell^Y\cup S_\ell^Z\ \forall \ell\in[w]}}
\mathrm{count}[Y,(\mathbf{S}^Y,\mathbf{a}^Y)]\cdot
\mathrm{count}[Z,(\mathbf{S}^Z,\mathbf{a}^Z)].
\end{equation}

Let us justify \eqref{eq:count-union-rigorous}.
Let $\varphi$ be any proper $k$-coloring of $G_X$.
Define $\varphi_Y\coloneqq \varphi|_{V_Y}$ and $\varphi_Z\coloneqq \varphi|_{V_Z}$.
Then $\varphi_Y$ and $\varphi_Z$ are proper colorings of $G_Y$ and $G_Z$, respectively, because
$G_Y$ and $G_Z$ are induced subgraphs of $G_X$ by \eqref{eq:union-structure}.
Conversely, for any pair $(\psi_Y,\psi_Z)$ where $\psi_Y$ is a proper coloring of $G_Y$ and $\psi_Z$ is a proper
coloring of $G_Z$, there is a unique coloring $\psi$ of $G_X$ obtained by combining them:
\[
\psi(v)=
\begin{cases}
\psi_Y(v) & v\in V_Y,\\
\psi_Z(v) & v\in V_Z,
\end{cases}
\]
and $\psi$ is proper since there are no edges between $V_Y$ and $V_Z$.
Now let $\sig_Y=(\mathbf{S}^Y,\mathbf{a}^Y)$ and $\sig_Z=(\mathbf{S}^Z,\mathbf{a}^Z)$ be the signatures represented
by $\varphi_Y$ and $\varphi_Z$, respectively.
We claim that $\varphi$ is represented by $\sig=(\mathbf{S},\mathbf{a})$ if and only if
\begin{equation}\label{eq:union-sig-cond}
\mathbf{a}=\mathbf{a}^Y+\mathbf{a}^Z
\quad\text{and}\quad
S_\ell=S_\ell^Y\cup S_\ell^Z\ \ \text{for all }\ell\in[w].
\end{equation}
Indeed, fix any color $c\in[k]$. Since $V_X=V_Y\uplus V_Z$, we have
\[
|\varphi^{-1}(c)| = |\varphi_Y^{-1}(c)| + |\varphi_Z^{-1}(c)|,
\]
hence $a_c=a_c^Y+a_c^Z$, which is exactly the $c$th coordinate of $\mathbf{a}=\mathbf{a}^Y+\mathbf{a}^Z$.

Next fix any label $\ell\in[w]$. Using \eqref{eq:union-structure}, the set of vertices of label $\ell$ in $G_X$ is the
disjoint union of the label-$\ell$ vertices in $G_Y$ and in $G_Z$, so the set of colors used by $\varphi$ on label
$\ell$ is exactly the union of the sets of colors used by $\varphi_Y$ and $\varphi_Z$ on label $\ell$, that is,
$S_\ell=S_\ell^Y\cup S_\ell^Z$.
This proves \eqref{eq:union-sig-cond}, and the converse direction is immediate from the definitions.

Finally, for every pair $(\sig_Y,\sig_Z)$ satisfying \eqref{eq:union-sig-cond}, the number of proper colorings
$\varphi$ of $G_X$ represented by $\sig$ whose restrictions induce these particular signatures equals
$\mathrm{count}[Y,\sig_Y]\cdot \mathrm{count}[Z,\sig_Z]$, since the choice of $\varphi_Y$ and the choice of
$\varphi_Z$ are independent and determine $\varphi$ uniquely.
Summing over all such pairs yields \eqref{eq:count-union-rigorous}.

\medskip
\noindent\textbf{Join.}
Assume $X=\join_{i,j}(Y)$ with $i\neq j$. Then $V_X=V_Y$, $\lab_X=\lab_Y$, and $G_X$ is obtained from $G_Y$ by
adding every edge between a vertex of label $i$ and a vertex of label $j$.
Fix an $X$-signature $\sig=(\mathbf{S},\mathbf{a})$. We set
\begin{equation}\label{eq:count-join-rigorous}
\mathrm{count}[X,\sig]
\coloneqq 
\begin{cases}
\mathrm{count}[Y,\sig] & \text{if } S_i\cap S_j=\emptyset,\\
0 & \text{otherwise.}
\end{cases}
\end{equation}

Let us justify \eqref{eq:count-join-rigorous}.
Let $\varphi$ be a proper $k$-coloring of $G_Y$ represented by $\sig$.
Then $\varphi$ is also a coloring of $G_X$ (same vertex set).
If $S_i\cap S_j=\emptyset$, then no color $c$ is used on both labels $i$ and $j$ by $\varphi$; hence for every new
edge $uv$ added by $\join_{i,j}$ (with $\lab_X(u)=i$ and $\lab_X(v)=j$) we have $\varphi(u)\neq \varphi(v)$.
Thus $\varphi$ remains proper in $G_X$. Conversely, if $S_i\cap S_j\neq\emptyset$, choose $c\in S_i\cap S_j$.
By definition of $S_i$ and $S_j$, there exist vertices $u,v$ with $\lab_X(u)=i$ and $\lab_X(v)=j$ and
$\varphi(u)=\varphi(v)=c$, and $\join_{i,j}$ adds the edge $uv$, contradicting properness in $G_X$.
Hence no proper coloring of $G_X$ can be represented by $\sig$, so $\mathrm{count}[X,\sig]=0$.
This proves \eqref{eq:count-join-rigorous}.

\medskip
\noindent\textbf{Relabel.}
Assume $X=\ren_{i\to j}(Y)$ with $i\neq j$. Then $V_X=V_Y$, $E(G_X)=E(G_Y)$, and $\lab_X$ is obtained from $\lab_Y$ by changing every label-$i$ vertex into label $j$.

Fix an $X$-signature $\sig'=(\mathbf{S}',\mathbf{a})$.
We set
\begin{equation}\label{eq:count-ren-rigorous}
\mathrm{count}[X,\sig']
\coloneqq 
\sum_{\sig:\ \sig\mapsto_{i\to j}\sig'} \mathrm{count}[Y,\sig],
\end{equation}
where $\sig\mapsto_{i\to j}\sig'$ is the relation of~\Cref{def:sig-map-relabel}.

Let us justify \eqref{eq:count-ren-rigorous}.
Since $E(G_X)=E(G_Y)$, a mapping $\varphi:V_X\to[k]$ is a proper coloring of $G_X$ if and only if it is a proper
coloring of $G_Y$. Thus relabeling does not change the set of proper colorings; it only changes which colors are
recorded in which label-set $S_\ell$.
For a fixed proper coloring $\varphi$, let $\sig$ be its $Y$-signature and $\sig'$ its $X$-signature.
Then necessarily $\mathbf{a}$ is the same in both signatures, and the label-sets satisfy exactly the update in~\Cref{def:sig-map-relabel}, i.e.\ $\sig\mapsto_{i\to j}\sig'$.
Therefore, for each fixed $\sig'$, the set of proper colorings of $G_X$ represented by $\sig'$ is the disjoint union
over all $Y$-signatures $\sig$ mapping to $\sig'$ of the sets of proper colorings represented by $\sig$ in $G_Y$.
Taking cardinalities yields \eqref{eq:count-ren-rigorous}.

\medskip

We compute $\mathrm{count}[X^\star,\sig]$ for all signatures $\sig$ of $X^\star$.
Let $n=kq+r$ with $0\le r<k$.
By~\Cref{lem:root-condition-cw}, a proper $k$-coloring of $G$ is equitable if and only if its size vector
$\mathbf{a}$ satisfies $a_c\in\{q,q+1\}$ for all $c\in[k]$ and exactly $r$ indices satisfy $a_c=q+1$.
Hence the number of equitable $k$-colorings of $G$ equals the sum of $\mathrm{count}[X^\star,\sig]$ over all
signatures $\sig=(\mathbf{S},\mathbf{a})$ of $X^\star$ with this property.

It remains to discuss the running time.
For fixed $k$, represent each $S_\ell\subseteq[k]$ by a $k$-bit mask, so unions and intersections take constant time.
The number of possible choices of $\mathbf{S}$ is $(2^k)^w$, and the number of possible vectors
$\mathbf{a}\in\{0,\dots,n\}^k$ is $(n+1)^k$. Thus for each subexpression $X$ the table contains at most
$M=(2^k)^w\cdot (n+1)^k$
entries.
For join nodes, each entry can be computed in constant time from the child table, hence these subexpressions can be processed in
$O(M)$ time.
For relabel nodes, fix a target signature $\sig'=(\mathbf{S}',\mathbf{a})$.
By \Cref{def:sig-map-relabel}, a preimage under $\mapsto_{i\to j}$ can exist only if $S'_i=\emptyset$; otherwise
$\mathrm{count}[X,\sig']=0$.
Assume therefore that $S'_i=\emptyset$.
Then the only possible preimages are obtained by choosing, for every color in $S'_j$, whether it came from old label $i$,
old label $j$, or both, while all other coordinates are forced.
Hence the number of preimages is at most $3^{|S'_j|}\le 3^k$, which is constant for fixed $k$.
Therefore relabel subexpressions also take $O(M)$ time.
For a union subexpression, we may evaluate \eqref{eq:count-union-rigorous} by iterating over all pairs of child signatures
(at most $M^2$ pairs), computing the unique combined signature, and adding the product of the two child counts; this
takes $O(M^2)$ time.

If $m$ denotes the size of the given expression $\mathcal{E}$, then the number of subexpressions of $\mathcal{E}$ is
$O(m)$. Therefore the arithmetic-operation count is $O(m\cdot M^2)=2^{O(k\cdot w)}\cdot m\cdot n^{O(k)}$.
Finally, $\mathrm{count}[X,\sig]\le k^{|V_X|}\le k^n$, so for fixed $k$ each table entry has $O(n)$ bits, and exact
integer arithmetic contributes only a polynomial factor in $n$.
This completes the proof of~\Cref{thm:cw-counting}.
\end{proof}

We deduce the following.

\begin{corollary}\label{cor:cw-decision}
There exists an algorithm that, given an integer $k\ge 1$ and an $n$-vertex graph $G$ together with a $w$-expression for $G$,
decides whether $G$ admits an equitable $k$-coloring in time
$2^{O(k \cdot w)}\cdot n^{O(k)}$,
assuming a standard polynomial-size encoding of the expression.
\end{corollary}

\begin{proof}
Compute the number of equitable $k$-colorings using~\Cref{thm:cw-counting} and accept if and only if the count
is nonzero.
\end{proof}

For the remainder of the section, we provide a $\mathsf{SETH}$-based lower bound for~\textsc{Equitable $k$-Coloring} on graphs of bounded clique-width. We first need the following result of~\cite{lampis2020finer}.

\begin{theorem}[Lampis~\cite{lampis2020finer}]\label{thm:lampis}
For every fixed $k\ge 3$ and every $\varepsilon>0$, if there exists an algorithm solving
\textsc{$k$-Coloring} in time
$O^*((2^k-2-\varepsilon)^{\cw(G)})$ on an input graph $G$,
then SETH is false.
\end{theorem}

\begin{theorem}\label{thm:cw-seth}
Let $G$ be a graph. For every fixed $k\ge 3$ and every $\varepsilon >0$, unless SETH fails,
\textsc{Equitable $k$-Coloring} cannot be solved in time
$O^\ast \big((2^k - 2 - \varepsilon)^{\cw(G)}\big),$
where $\cw(G)$ denotes the clique-width of $G$.
\end{theorem}

\begin{proof}
We proceed by a reduction from \textsc{$k$-Coloring} that increases clique-width by at most one.
Suppose for a contradiction that there exists an algorithm $\mathcal{A}$ which, on input a graph $H$,
decides \textsc{Equitable $k$-Coloring} in time
$O^{\ast}((2^k-2-\varepsilon)^{\cw(H)})$.
Let $G$ be an instance of \textsc{$k$-Coloring} and put $n\coloneqq |V(G)|$.
Let $I$ be an independent set on a fresh vertex set $U$ with $|U|=(k-1)n$, and define
$G' \coloneqq  G \uplus I.$
Note that $V(G')=V(G)\uplus U$ and $E(G')=E(G)$.

\begin{claim}\label{clm:pad-equitable}
$G$ is properly $k$-colorable if and only if $G'$ admits an equitable proper $k$-coloring.
\end{claim}

\begin{claimproof}
We first prove the forward implication.
Assume that $G$ has a proper $k$-coloring $\psi:V(G)\to[k]$.
For each color $c\in[k]$ let $s_c\coloneqq |\psi^{-1}(c)|$, so $\sum_{c=1}^k s_c=n$.
Choose a partition $U=\biguplus_{c=1}^k U_c$ such that $|U_c|=n-s_c$ for every $c\in[k]$.
This is possible because
\[
\sum_{c=1}^k (n-s_c) = kn - \sum_{c=1}^k s_c = kn-n = (k-1)n = |U|.
\]
Define $\psi':V(G')\to[k]$ by $\psi'|_{V(G)}=\psi$ and $\psi'(u)=c$ for all $u\in U_c$.
We show that $\psi'$ is an equitable proper $k$-coloring of $G'$.
Observe that every edge of $G'$ lies inside $V(G)$, since $U$ is independent and there are no edges between $V(G)$ and $U$.
Thus for every edge $xy\in E(G')=E(G)$ we have $\psi'(x)=\psi(x)\neq \psi(y)=\psi'(y)$, because $\psi$ is proper.
Moreover, for each $c\in[k]$ we have
\[
|\psi'^{-1}(c)| = |\psi^{-1}(c)| + |U_c| = s_c + (n-s_c) = n,
\]
so all $k$ color classes have the same size, hence $\psi'$ is equitable.

For the backward implication, assume that $G'$ admits an equitable proper $k$-coloring $\varphi:V(G')\to[k]$.
Consider the restriction $\varphi|_{V(G)}:V(G)\to[k]$.
Since $E(G)=E(G')\cap \binom{V(G)}{2}$, every edge $xy\in E(G)$ is also an edge of $G'$.
Therefore $\varphi(x)\neq \varphi(y)$ for all $xy\in E(G)$, because $\varphi$ is proper on $G'$.
Hence $\varphi|_{V(G)}$ is a proper $k$-coloring of $G$.
This proves the claim.
\end{claimproof}

\begin{claim}\label{clm:cw-padding}
$\cw(G') \le \cw(G)+1$.
\end{claim}

\begin{claimproof}
Let $w\coloneqq \cw(G)$. By definition of clique-width (see~\Cref{sec:prelim}), there exists a $w$-expression $\mathcal{E}$ that constructs $G$.
We construct a $(w+1)$-expression $\mathcal{E}'$ that constructs $G'$.
Let $\ell^\star\coloneqq w+1$ be a new label not used in $\mathcal{E}$.
Starting from $\mathcal{E}$, we append the following operation $(k-1)n$ times:
create a new single-vertex $(w+1)$-labelled graph whose unique vertex has label $\ell^\star$,
and take the disjoint union of the current graph with this vertex via $\union(\cdot,\cdot)$.
During this extension we never apply any join operation $\join_{\ell^\star,\ell}$ involving $\ell^\star$, and we never
relabel $\ell^\star$.
We claim that the resulting expression constructs $G\uplus I$.
Indeed, disjoint union introduces no edges between the previously constructed vertices and the new vertex.
Since we never apply a join operation involving $\ell^\star$, the newly created vertices never gain incident edges,
so they remain isolated. Thus the vertices created with label $\ell^\star$ induce an independent set of size $(k-1)n$,
and the induced subgraph on the original vertex set is exactly $G$, since we started from $\mathcal{E}$ and did not
change edges among the original vertices.
Therefore $\mathcal{E}'$ constructs $G'$, using at most $w+1$ labels, and hence $\cw(G')\le w+1=\cw(G)+1$.
\end{claimproof}

We are now ready to complete the reduction.
Given $G$, we construct $G'$ in polynomial time and run $\mathcal{A}$ on $G'$.
By Claim~\ref{clm:pad-equitable}, $\mathcal{A}$ answers YES on $G'$ if and only if $G$ is $k$-colorable.
Moreover, by Claim~\ref{clm:cw-padding},
$\cw(G') \le \cw(G)+1.$
Hence the running time of $\mathcal{A}$ on $G'$ is
\[
O^\ast\!\bigl((2^k-2-\varepsilon)^{\cw(G')}\bigr)
\;\le\;
O^\ast\!\bigl((2^k-2-\varepsilon)^{\cw(G)+1}\bigr).
\]
Since for fixed $k$ and $\varepsilon>0$, the factor $(2^k-2-\varepsilon)$ is a constant absorbed by $O^*$, it follows that $\mathcal{A}$ would yield an algorithm for \textsc{$k$-Coloring} running in time
\[
O^\ast\!\bigl((2^k-2-\varepsilon)^{\cw(G)}\bigr).
\]
This contradicts the result of~\Cref{thm:lampis}, and hence completes the proof of~\Cref{thm:cw-seth}.
\end{proof}


\section{Graphs of bounded Linear clique-width: The Proof of~\Cref{thm:lcw-counting-restate}}\label{sec:lcw}

We now prove our result for graphs of bounded linear clique-width.
To design the dynamic programming algorithm, we process a given linear $w$-expression from left to right.
Accordingly, at each step we need to refer not only to the final graph produced by the whole expression, but also to the
labelled graph obtained after executing an initial prefix of the expression.
These intermediate labelled graphs allow us to describe how labels evolve over time, which labels may still participate in
future join operations, and which partial colorings can still be extended to a proper coloring of the final graph.

Any linear $w$-expression can be viewed equivalently as a sequence of operations evaluated from left to right, starting
from the empty $w$-labelled graph.
Let us distinguish this sequential representation from the term representation used earlier for clique-width expressions, since here the left-to-right order of the construction is essential for the dynamic programming algorithm. In this sequential representation, the creation of a new single vertex of label $i$ is denoted by $\mathsf{add}(i)$, while
the operations corresponding to $\join_{i,j}$ and $\ren_{i\to j}$ are denoted by $\eta(i,j)$ and $\rho(i\to j)$,
respectively.
We write such a sequence as
$\Pi=\pi_1\pi_2\cdots \pi_L.$

\subsection*{Intermediate graphs}

Let $\Pi=\pi_1\pi_2\cdots \pi_L$ be a fixed linear $w$-expression.
For each $t\in\{0,1,\dots,L\}$, let
\[
\mathbf{F}_t=(F_t;C_1^{(t)},\dots,C_w^{(t)})
\]
denote the intermediate labelled graph obtained after executing the prefix $\pi_1\cdots\pi_t$,
with $\mathbf{F}_0$ the empty $w$-labelled graph.
Thus $\mathbf{F}_L$ is the $w$-labelled graph produced by the whole expression $\Pi$.
We write $V_t\coloneqq V(F_t)$, $E_t\coloneqq E(F_t)$, and $n_t\coloneqq |V_t|$.

\begin{definition}\label{def:label-evolution}
For distinct $p,q\in[w]$, let $r_{p\to q}:[w]\to[w]$ be the map defined by
\[
r_{p\to q}(p)=q,
\qquad
r_{p\to q}(x)=x \text{ for all }x\neq p.
\]
For $t<s$, define the map $\tau_{t,s}:[w]\to[w]$ by
\[
\tau_{t,s}\coloneqq f_s\circ f_{s-1}\circ\cdots\circ f_{t+1},
\]
where for each $r\in\{t+1,\dots,s\}$,
\[
f_r=
\begin{cases}
r_{p\to q} & \text{if }\pi_r=\rho(p\to q),\\
\mathrm{id}_{[w]} & \text{if }\pi_r=\mathsf{add}(\cdot)\text{ or }\pi_r=\eta(\cdot,\cdot).
\end{cases}
\]
In other words, $\tau_{t,s}(i)$ is the label at time $s$ of a vertex that had label $i$ at time $t$. We also let $\tau_{t,t}\coloneqq  \mathrm{id}_{[w]}$.
\end{definition}

\begin{lemma}\label{lem:tau-correct}
Fix $t<s$ and let $v\in V_t$ be any vertex.
If $v\in C_i^{(t)}$, then in $\mathbf{F}_s$ the vertex $v$ has label $\tau_{t,s}(i)$, that is,
$v\in C^{(s)}_{\tau_{t,s}(i)}$.
\end{lemma}

\begin{proof}
We prove the statement by induction on $s-t$.
If $s=t+1$, then $\pi_{t+1}$ is either $\mathsf{add}(\cdot)$ or $\eta(\cdot,\cdot)$, which do not change the labels of
existing vertices, or $\rho(p\to q)$, which changes the label exactly when $i=p$.
This is exactly the effect of $\tau_{t,t+1}$.

For the induction step, write $s'=s-1$.
By the induction hypothesis, in $\mathbf{F}_{s'}$ the label of $v$ is $\tau_{t,s'}(i)$.
If $\pi_s$ is $\mathsf{add}(\cdot)$ or $\eta(\cdot,\cdot)$, then labels do not change and
$\tau_{t,s}=\tau_{t,s'}$.
If $\pi_s=\rho(p\to q)$, then $v$ changes label exactly when its label in $\mathbf{F}_{s'}$ is $p$, that is,
when $\tau_{t,s'}(i)=p$, and in that case its new label becomes $q$.
This is precisely the effect of composing $\tau_{t,s'}$ with the relabel map $r_{p\to q}$, which is how
$\tau_{t,s}$ is defined.
\end{proof}

\begin{definition}[Live labels]\label{def:live-label}
Let $t\in\{0,\dots,L\}$ and $i\in[w]$.
We say that label $i$ is \emph{live at time $t$} if $C_i^{(t)}\neq\emptyset$ and there exists
$s\in\{t+1,\dots,L\}$ such that $\pi_s=\eta(p,q)$ with $p\neq q$ and either
\[
\tau_{t,s-1}(i)=p,\qquad C_q^{(s-1)}\neq\emptyset,
\]
or
\[
\tau_{t,s-1}(i)=q,\qquad C_p^{(s-1)}\neq\emptyset.
\]
We write $\Live(t)\subseteq[w]$ for the set of live labels at time $t$.
\end{definition}

Intuitively, a label $i$ is live at time $t$ precisely when the set of colors currently used on vertices of label $i$ may still affect properness at some future join operation; labels that are not live can therefore be ignored by the dynamic programming algorithm.

\begin{claim}\label{clm:compute-live}
Given a linear $w$-expression $\Pi=\pi_1\cdots\pi_L$, one can compute $\Live(t)$ for all
$t\in\{0,\dots,L\}$ in time $O(L^2w^2)$.
\end{claim}

\begin{claimproof}
We first compute all label-class sizes $|C_i^{(t)}|$ for all $t$ and $i$ by simulating the expression forward.
This takes $O(Lw)$ time.
Fix $t\in\{0,\dots,L\}$.
We compute $\Live(t)$ by scanning the future steps $s=t+1,\dots,L$.
At the moment we process step $s$, we maintain a map $\mu:[w]\to[w]$ satisfying
$\mu=\tau_{t,s-1}$.
Initially, $\mu$ is the identity map on $[w]$.

If $\pi_s=\rho(p\to q)$, then
\[
\tau_{t,s}=r_{p\to q}\circ \tau_{t,s-1}.
\]
Hence we update $\mu$ to $r_{p\to q}\circ \mu$, equivalently:
for all $i\in[w]$ with $\mu(i)=p$, set $\mu(i)\leftarrow q$.
If $\pi_s=\eta(p,q)$ with $p\neq q$, then $\eta(p,q)$ does not affect $\mu$.

Whenever $\pi_s=\eta(p,q)$, for each label $i$ with $C_i^{(t)}\neq\emptyset$ we check whether
\[
\mu(i)=p \text{ and } C_q^{(s-1)}\neq\emptyset,
\]
or
\[
\mu(i)=q \text{ and } C_p^{(s-1)}\neq\emptyset.
\]
If so, we add $i$ to $\Live(t)$.

For fixed $t$, there are $O(L)$ future steps, and at each step we may inspect all $w$ labels and update or test a map on
$[w]$, so the running time is $O(Lw^2)$ for this fixed $t$.
Over all $t\in\{0,\dots,L\}$, the total running time is therefore $O(L^2w^2)$.
\end{claimproof}

\begin{definition}\label{def:extendible}
A proper $k$-coloring $\varphi_t:V_t\to[k]$ of $F_t$ is \emph{extendible} if there exists a proper $k$-coloring
$\varphi:V_L\to[k]$ of the graph $F_L$ produced by the whole linear expression $\Pi$ such that
$\varphi|_{V_t}=\varphi_t$.
\end{definition}

\begin{lemma}\label{lem:live-proper-subset}
Let $t\in\{0,\dots,L\}$ and let $\varphi_t$ be an extendible proper $k$-coloring of $F_t$.
Then for every live label $i\in\Live(t)$ the set of colors used on $C_i^{(t)}$,
\[
S_i(\varphi_t)\coloneqq \{\,c\in[k] : \exists v\in C_i^{(t)} \text{ with }\varphi_t(v)=c\,\},
\]
satisfies $S_i(\varphi_t)\neq\emptyset$ and $S_i(\varphi_t)\neq [k]$.
\end{lemma}

\begin{proof}
Since $i\in\Live(t)$, we have $C_i^{(t)}\neq\emptyset$, hence $S_i(\varphi_t)\neq\emptyset$.
Let $s>t$ and $\eta(p,q)=\pi_s$ witness that $i$ is live.
Without loss of generality assume
\[
\tau_{t,s-1}(i)=p
\qquad\text{and}\qquad
C_q^{(s-1)}\neq\emptyset.
\]
By~\Cref{lem:tau-correct}, every vertex of $C_i^{(t)}$ belongs to $C_p^{(s-1)}$.
Let $\varphi$ be a proper $k$-coloring of $F_L$ extending $\varphi_t$, and set
$\varphi_{s-1}\coloneqq \varphi|_{V_{s-1}}$.
For any label $\ell\in[w]$, write
\[
S_\ell(\varphi_{s-1})
\coloneqq 
\{\,c\in[k]: \exists v\in C_\ell^{(s-1)} \text{ with } \varphi_{s-1}(v)=c\,\}.
\]
Since $\pi_s=\eta(p,q)$ adds all edges between $C_p^{(s-1)}$ and $C_q^{(s-1)}$, properness of $\varphi$ implies
\[
S_p(\varphi_{s-1})\cap S_q(\varphi_{s-1})=\emptyset.
\]
Because $C_q^{(s-1)}\neq\emptyset$, we have $S_q(\varphi_{s-1})\neq\emptyset$, and hence
$S_p(\varphi_{s-1})\neq [k]$.
Finally, $S_i(\varphi_t)\subseteq S_p(\varphi_{s-1})$, so $S_i(\varphi_t)\neq [k]$.
\end{proof}

\begin{claim}\label{clm:live-stability-add}
If $\pi_t=\mathsf{add}(i)$, then
\[
\Live(t)\setminus\{i\}=\Live(t-1)\setminus\{i\}.
\]
\end{claim}

\begin{claimproof}
Fix $\ell\in[w]\setminus\{i\}$.
The operation $\mathsf{add}(i)$ does not change $C_\ell$ and does not affect the label evolution of already present
vertices.
Hence $C_\ell^{(t)}\neq\emptyset$ if and only if $C_\ell^{(t-1)}\neq\emptyset$, and for every $s\ge t+1$ we have
\[
\tau_{t,s-1}(\ell)=\tau_{t-1,s-1}(\ell).
\]
Therefore the existence of a witness join step in~\Cref{def:live-label} is identical at times $t-1$ and $t$
for label $\ell$.
\end{claimproof}

\begin{claim}\label{clm:add-live}
If $\pi_t=\mathsf{add}(i)$ and $i\in\Live(t)$ but $i\notin\Live(t-1)$, then
$C_i^{(t-1)}=\emptyset.$
\end{claim}

\begin{claimproof}
Suppose for a contradiction that $C_i^{(t-1)}\neq\emptyset$.
Since $i\in\Live(t)$, let $s>t$ be a witness join step for $i$ at time $t$.
The operation $\mathsf{add}(i)$ does not affect the label evolution of vertices already present, so
$
\tau_{t-1,s-1}(i)=\tau_{t,s-1}(i).
$
Together with $C_i^{(t-1)}\neq\emptyset$ and the same witness condition on the opposite side at time $s-1$, the same step
$s$ witnesses $i\in\Live(t-1)$, a contradiction.
\end{claimproof}

\begin{claim}\label{clm:join-live-monotone}
If $\pi_t=\eta(p,q)$, then
$\Live(t)\subseteq \Live(t-1).$
\end{claim}

\begin{claimproof}
Fix $\ell\in\Live(t)$.
Then $C_\ell^{(t)}\neq\emptyset$ and there exists $s>t$ witnessing liveness for $\ell$ at time $t$.
The operation $\eta(p,q)$ does not change labels or label classes, so
$
C_\ell^{(t-1)}=C_\ell^{(t)}\neq\emptyset.
$
Moreover, since $\eta(p,q)$ does not affect label evolution,
$
\tau_{t-1,s-1}(\ell)=\tau_{t,s-1}(\ell).
$
Thus the same witness $s$ also shows $\ell\in\Live(t-1)$.
\end{claimproof}

\begin{claim}\label{clm:ren-live-stability}
If $\pi_t=\rho(p\to q)$, then for every $\ell\in[w]\setminus\{p,q\}$ we have
\[
\ell\in\Live(t) \text{ if and only if } \ell\in\Live(t-1).
\]
\end{claim}

\begin{claimproof}
Fix $\ell\notin\{p,q\}$.
The relabel $\rho(p\to q)$ does not affect $C_\ell$ and does not change the label evolution of vertices in $C_\ell$.
Hence $C_\ell^{(t)}\neq\emptyset$ if and only if $C_\ell^{(t-1)}\neq\emptyset$, and for every $s\ge t+1$ we have
$
\tau_{t,s-1}(\ell)=\tau_{t-1,s-1}(\ell).
$
Therefore $\ell$ has a witness join step after $t$ if and only if it has one after $t-1$.
\end{claimproof}

\begin{claim}\label{clm:q-becomes-live-empty}
If $\pi_t=\rho(p\to q)$ and $q\in\Live(t)$ but $q\notin\Live(t-1)$, then
$
C_q^{(t-1)}=\emptyset.
$
\end{claim}

\begin{claimproof}
Suppose for a contradiction that $C_q^{(t-1)}\neq\emptyset$.
Since $q\in\Live(t)$, let $s>t$ be a witness join step for $q$ at time $t$.
The relabel $\rho(p\to q)$ does not change the label of vertices already in $C_q^{(t-1)}$, hence for those vertices
$
\tau_{t-1,s-1}(q)=\tau_{t,s-1}(q).
$
Together with $C_q^{(t-1)}\neq\emptyset$ and the same witness condition on the opposite side, the same step $s$ witnesses
$q\in\Live(t-1)$, a contradiction.
\end{claimproof}

\begin{claim}\label{clm:relabel-live-source}
If $\pi_t=\rho(p\to q)$, $q\in\Live(t)$, and $C_p^{(t-1)}\neq\emptyset$, then
$
p\in\Live(t-1).
$
\end{claim}

\begin{claimproof}
Let $q\in\Live(t)$ and let $s>t$ be a witness join step for $q$ at time $t$.
Choose any vertex $v\in C_p^{(t-1)}$.
After applying $\rho(p\to q)$ at step $t$, the vertex $v$ belongs to $C_q^{(t)}$.
Hence at time $s-1$ its label equals $\tau_{t,s-1}(q)$.
By the definition of the evolution maps,
$
\tau_{t-1,s-1}(p)=\tau_{t,s-1}(q).
$
Since $C_p^{(t-1)}\neq\emptyset$ and the witness requires the opposite label class at time $s-1$ to be nonempty, the same
step $s$ witnesses $p\in\Live(t-1)$.
\end{claimproof}

We are now ready to prove the main result of this section.

\begin{theorem}\label{thm:lcw-counting-2k-2}
There exists an algorithm that, given an integer $k\ge 1$ and an $n$-vertex graph $G$ together with a linear $w$-expression constructing $G$, computes the number of equitable $k$-colorings of $G$ in time
$\max\{1,2^k-2\}^w\cdot n^{k+O(1)}$.
In particular, for every fixed integer $k$, counting equitable $k$-colorings is polynomial-time solvable on graph classes of bounded linear clique-width, given a linear clique-width expression.
\end{theorem}

\begin{proof}
We first handle the case $k=1$ separately.
A proper $1$-coloring of $G$ exists if and only if $G$ has no edges, and in that case it is unique.
Thus the number of equitable $1$-colorings is $1$ if $E(G)=\emptyset$, and $0$ otherwise.
This can be checked in polynomial time from the given linear expression.
In the remainder of the proof, assume that $k\ge 2$.
Let $\Pi=\pi_1\pi_2\cdots \pi_L$ be the given linear $w$-expression, where $L$ is its length.
Let $\mathbf{F}_t=(F_t;C_1^{(t)},\dots,C_w^{(t)})$, for $t=0,1,\dots,L$, be the intermediate labelled graphs,
with $F_0$ empty and $F_L=G$.
Let $V_t\coloneqq V(F_t)$ and $n_t\coloneqq |V_t|$.
By Claim~\ref{clm:compute-live}, we may assume that $\Live(t)$ is known for all $t$.
We use a distinguished symbol $\bot$ to denote that a label is not tracked.
For each $t\in\{0,\dots,L\}$, a \emph{state at time $t$} is a pair $\sig=(\mathbf{S},\mathbf{a})$ where
\[
\mathbf{a}=(a_1,\dots,a_k)\in\{0,1,\dots,n_t\}^k,
\qquad
\sum_{c=1}^k a_c = n_t,
\]
and $\mathbf{S}=(S_1,\dots,S_w)$ satisfies
\[
S_i =
\begin{cases}
\bot & \text{if } i\notin \Live(t),\\[1mm]
\text{an element of }2^{[k]}\setminus\{\emptyset,[k]\} & \text{if } i\in \Live(t).
\end{cases}
\]

For a proper $k$-coloring $\varphi_t:V_t\to[k]$ of $F_t$, its \emph{projected signature at time $t$} is the pair
$\sig=(\mathbf{S},\mathbf{a})$ defined by
\[
a_c \coloneqq  |\varphi_t^{-1}(c)| \quad (c\in[k]),
\qquad
S_i \coloneqq 
\begin{cases}
\{c\in[k]:\exists v\in C_i^{(t)} \text{ with }\varphi_t(v)=c\} & \text{if } i\in\Live(t),\\
\bot & \text{if } i\notin\Live(t).
\end{cases}
\]

We now construct the DP table.
For each $t$ and each state $\sig$ at time $t$, let
\[
\DP[t,\sig]\in\mathbb{Z}_{\ge 0}
\]
be the number of proper $k$-colorings $\varphi_t$ of $F_t$ whose projected signature equals $\sig$.

We do not store colorings whose projected signature is not a state at time $t$; equivalently, we discard colorings for
which some live label uses all $k$ colors.
By~\Cref{lem:live-proper-subset}, such colorings are not extendible to a proper coloring of $F_L$, and hence
discarding them cannot affect the final count.
Moreover, once a coloring is discarded for this reason, it can never later contribute to a valid state:
if some live label uses all $k$ colors at time $t$, then under subsequent add and relabel operations the corresponding
live descendant label still uses all $k$ colors, while when the witnessing future join is performed the coloring becomes
improper.

At $t=0$ the graph is empty, so $\Live(0)=\emptyset$.
There is exactly one proper coloring, namely the empty map, whose projected signature is the unique state $\sig_0$ with
$\mathbf{a}=\mathbf{0}$ and $S_i=\bot$ for all $i$.
Set
$\DP[0,\sig_0]=1$ and 
$\DP[0,\sig]=0 \text{ for all other states }\sig.$

Assume that $\DP[t-1,\cdot]$ has been computed.
Before processing step $t$, initialize
$\DP[t,\sig]\coloneqq 0$
for every state $\sig$ at time $t$.
We now describe the DP update from time $t-1$ to time $t$, depending on the operation $\pi_t$.

\paragraph{Case 1: $\pi_t=\mathsf{add}(i)$.}
Let $v$ be the new vertex, so
$
V_t=V_{t-1}\uplus\{v\}$ and 
$v\in C_i^{(t)}.$
Fix a state $\sig^-=(\mathbf{S}^-,\mathbf{a}^-)$ with $\DP[t-1,\sig^-]>0$.
For each color $c\in[k]$, define $\mathbf{a}^+$ by
$a_c^+=a_c^-+1$ and 
$a_{c'}^+=a_{c'}^- \text{ for } c'\neq c$.
Define $\mathbf{S}^+$ by:
\begin{itemize}
\item for every $\ell\in\Live(t)\setminus\{i\}$, set
$
S_\ell^+\coloneqq S_\ell^-,
$
which is well-defined by Claim~\ref{clm:live-stability-add};
\item for every $\ell\notin\Live(t)$, set
$
S_\ell^+\coloneqq \bot;
$
\item for label $i$, set
\[
S_i^+\coloneqq 
\begin{cases}
S_i^- \cup \{c\} & \text{if } i\in\Live(t)\cap\Live(t-1),\\
\{c\} & \text{if } i\in\Live(t)\setminus\Live(t-1),\\
\bot & \text{if } i\notin\Live(t),
\end{cases}
\]
where the second case is valid by Claim~\ref{clm:add-live}.
\end{itemize}
If $i\in\Live(t)$ and $S_i^+=[k]$, discard this update.
Otherwise set
\[
\DP[t,(\mathbf{S}^+,\mathbf{a}^+)]
\coloneqq 
\DP[t,(\mathbf{S}^+,\mathbf{a}^+)]
+
\DP[t-1,(\mathbf{S}^-,\mathbf{a}^-)].
\]

This is correct since proper colorings of $F_t$ are in bijection with pairs consisting of a proper coloring of
$F_{t-1}$ and a color for the new isolated vertex $v$, and the projected signature changes exactly as above.

\paragraph{Case 2: $\pi_t=\eta(p,q)$ with $p\neq q$.}
If $C_p^{(t-1)}=\emptyset$ or $C_q^{(t-1)}=\emptyset$, then $F_t=F_{t-1}$.
By Claim~\ref{clm:join-live-monotone}, we have
$
\Live(t)\subseteq\Live(t-1).
$
For each state $\sig^-=(\mathbf{S}^-,\mathbf{a})$, define $\sig^+=(\mathbf{S}^+,\mathbf{a})$ by
\[
S_\ell^+=S_\ell^- \text{ for } \ell\in\Live(t),
\qquad
S_\ell^+=\bot \text{ otherwise},
\]
and update
\[
\DP[t,\sig^+] \coloneqq  \DP[t,\sig^+] + \DP[t-1,\sig^-].
\]
This simply projects the same proper coloring onto the smaller live set.

Assume now that $C_p^{(t-1)}\neq\emptyset$ and $C_q^{(t-1)}\neq\emptyset$.
Then $p,q\in\Live(t-1)$, because the current join step itself witnesses liveness at time $t-1$.
Fix a state $\sig^-=(\mathbf{S}^-,\mathbf{a})$ with $\DP[t-1,\sig^-]>0$.
If
$
S_p^-\cap S_q^-\neq\emptyset,
$
add nothing.
If
$
S_p^-\cap S_q^-=\emptyset,
$
define $\sig^+=(\mathbf{S}^+,\mathbf{a})$ by
\[
S_\ell^+=S_\ell^- \text{ for } \ell\in\Live(t),
\qquad
S_\ell^+=\bot \text{ otherwise},
\]
and update
\[
\DP[t,\sig^+] \coloneqq  \DP[t,\sig^+] + \DP[t-1,\sig^-].
\]
Indeed, when both label classes are nonempty, a proper coloring survives the join if and only if no color is used on
both sides, that is, if and only if $S_p^-\cap S_q^-=\emptyset$.

\paragraph{Case 3: $\pi_t=\rho(p\to q)$ with $p\neq q$.}
Fix a state $\sig^-=(\mathbf{S}^-,\mathbf{a})$ with $\DP[t-1,\sig^-]>0$.
The relabel operation does not change the graph $F_{t-1}$, only the labels, so the color-count vector remains
$\mathbf{a}$.
For every $\ell\in\Live(t)\setminus\{q\}$, set
$S_\ell^+\coloneqq S_\ell^-.$
This is well-defined for $\ell\notin\{p,q\}$ by Claim~\ref{clm:ren-live-stability}, and label $p$ cannot belong to
$\Live(t)$ because after the relabel its class is empty.
For every $\ell\notin\Live(t)$, set
$S_\ell^+\coloneqq \bot.$

For label $q$, define
\[
S_q^+\coloneqq 
\begin{cases}
\bot & \text{if } q\notin\Live(t),\\[1mm]
S_q^- \cup S_p^- & \text{if } q\in\Live(t)\cap\Live(t-1)\text{ and }p\in\Live(t-1),\\[1mm]
S_q^- & \text{if } q\in\Live(t)\cap\Live(t-1)\text{ and }p\notin\Live(t-1),\\[1mm]
S_p^- & \text{if } q\in\Live(t)\setminus\Live(t-1).
\end{cases}
\]
These cases are exhaustive: if $q\notin\Live(t)$, we are in the first case; otherwise either
$q\in\Live(t-1)$ or $q\notin\Live(t-1)$, and in the former situation either $p\in\Live(t-1)$ or
$p\notin\Live(t-1)$.
In the third case, if $C_p^{(t-1)}\neq\emptyset$, then Claim~\ref{clm:relabel-live-source} would imply
$p\in\Live(t-1)$, a contradiction.
Hence $C_p^{(t-1)}=\emptyset$, so no vertices are moved from label $p$ to label $q$, and therefore
$S_q^+=S_q^-$.
In the last case, Claim~\ref{clm:q-becomes-live-empty} gives
$C_q^{(t-1)}=\emptyset$.
Since $q\in\Live(t)$, we have $C_q^{(t)}\neq\emptyset$ by definition of liveness.
Because $\rho(p\to q)$ relabels all vertices of $C_p^{(t-1)}$ to label $q$, we have
$C_q^{(t)}=C_q^{(t-1)}\cup C_p^{(t-1)}=C_p^{(t-1)}$,
and therefore $C_p^{(t-1)}\neq\emptyset$.
Now Claim~\ref{clm:relabel-live-source} implies $p\in\Live(t-1)$, so $S_p^-$ is defined.

If $q\in\Live(t)$ and $S_q^+=[k]$, discard this update.
Otherwise set
$\DP[t,(\mathbf{S}^+,\mathbf{a})]
:=
\DP[t,(\mathbf{S}^+,\mathbf{a})]
+
\DP[t-1,(\mathbf{S}^-,\mathbf{a})].
$

This is correct because relabeling induces a bijection between proper colorings before and after the operation, and the
projected signature changes exactly as above.
Any update yielding $S_q^+=[k]$ can be discarded by~\Cref{lem:live-proper-subset}, since such a coloring cannot
extend to a proper coloring of $F_L$.

\medskip

This completes the case analysis for the DP update.
By induction on $t$, for every state $\sig$ at time $t$, the table value $\DP[t,\sig]$ equals the number of proper
$k$-colorings of $F_t$ whose projected signature equals $\sig$.

At time $L$, we have $\Live(L)=\emptyset$, so every state has the form
$
(\bot,\dots,\bot;\mathbf{a}),
$
and the projected signature records only the color-class sizes.
For $n=kq+r$ with $0\le r<k$, the number of equitable proper $k$-colorings of $G=F_L$ is obtained by summing
$\DP[L,\sig]$ over all states whose vector $\mathbf a$ satisfies
$a_c\in\{q,q+1\}$ with $c\in[k]$
and exactly $r$ coordinates are equal to $q+1$.

For the running time, first consider the case $k\ge 2$.
At each time $t$, there are at most
$(2^k-2)^{|\Live(t)|}\le (2^k-2)^w$
choices for the tracked label-sets and at most $(n+1)^k$ choices for $\mathbf{a}$, so the number of states is at most
$(2^k-2)^w(n+1)^k.$
Each state update requires only polynomially many operations on labels and on subsets of $[k]$.
Representing subsets of $[k]$ by $k$-bit vectors, every union, intersection, equality test, and membership test can be carried out in time polynomial in $k$.
Hence each transition can be processed in time polynomial in $k$ and $w$.
Since a linear $w$-expression that constructs an $n$-vertex graph uses at most $n$ labels, we have $w\le n$, and therefore every polynomial factor in $w$ is absorbed into $n^{O(1)}$.
Moreover, for $k\ge 2$ and $w\ge 1$, the factor $(2^k-2)^w$ dominates every polynomial factor in $k$.
Thus, up to polynomial factors in $n$, the total running time is
$(2^k-2)^w(n+1)^k=
(2^k-2)^w\cdot n^{k+O(1)}.$
Together with the polynomial-time treatment of the case $k=1$, this yields the claimed bound
$\max\{1,2^k-2\}^w\cdot n^{k+O(1)}$. This completes the proof of~\Cref{thm:lcw-counting-2k-2}.
\end{proof}

\begin{corollary}\label{cor:lcw-decision-2k-2}
There exists an algorithm that, given an integer $k\ge 1$ and an $n$-vertex graph $G$ together with a linear $w$-expression producing $G$, decides whether $G$ admits an equitable $k$-coloring in time
$\max\{1,2^k-2\}^w\cdot n^{k+O(1)}$.
In particular, for every fixed integer $k$, deciding whether a graph admits an equitable $k$-coloring is polynomial-time solvable on graph classes of bounded linear clique-width, given a linear clique-width expression.
\end{corollary}

\begin{proof}
Run the algorithm of~\Cref{thm:lcw-counting-2k-2} to compute the number of equitable proper $k$-colorings of $G$
and accept if and only if the computed number is nonzero.
\end{proof}

\section{Graphs with no long induced paths: The Proof of~\Cref{thm:PT-restate}}\label{sec:pt}

The goal of this section is to prove the following.

\begin{theorem}\label{thm:PT}
Let $t\in \mathbb{N}$. There is an algorithm that, given a $P_t$-free graph $G$ on $n$ vertices and a list assignment
$L\colon V(G)\to 2^{\{1,2,3\}}$, computes the number of equitable list $3$-colorings of $(G,L)$ in time
$2^{O(\sqrt{n\log n})}\cdot \mathrm{poly}(n)\footnote{Since $t$ is fixed, the factor depending on $t$ is absorbed in the $O(\cdot)$-notation.}$.
\end{theorem}

We first need a few definitions. Recall that a list $3$-coloring $\varphi$ of $(G,L)$ is \emph{equitable} if for all
$i,j\in\{1,2,3\}$,
$
\bigl||\varphi^{-1}(i)|-|\varphi^{-1}(j)|\bigr|\le 1.
$
Equivalently, writing $n=3q+r$ with $r\in\{0,1,2\}$, the multiset
$\{|\varphi^{-1}(1)|,|\varphi^{-1}(2)|,|\varphi^{-1}(3)|\}$ equals
$\{q,q,q\}$ if $r=0$,
$\{q+1,q,q\}$ if $r=1$,
and $\{q+1,q+1,q\}$ if $r=2$.

Let $H$ be a multigraph and $G$ a graph.
A \emph{graph homomorphism} from $G$ to $H$ is a map $f\colon V(G)\to V(H)$ such that whenever $uv\in E(G)$,
we have $f(u)f(v)\in E(H)$.
Thus, a $3$-coloring is precisely a graph homomorphism to $K_3$.
A \emph{list $H$-coloring instance} is a pair $(G,L)$, where $L\colon V(G)\to 2^{V(H)}$ assigns to every
vertex $v\in V(G)$ a list $L(v)\subseteq V(H)$.
A \emph{list $H$-coloring} of $(G,L)$ is a graph homomorphism $f$ from $G$ to $H$ such that
$f(v)\in L(v)$ for all $v\in V(G)$.
Following~\cite{groenland2019h}, we denote the set of list $H$-colorings of $(G,L)$ by
$\mathcal{LC}((G,L),H)$.

For a given multigraph $H$, we define the partition polynomial
$p_{(G,L)\to H}(x)$ by
\begin{equation}\label{eq:partition-poly}
p_{(G,L)\to H}(x)
\coloneqq 
\sum_{f\in \mathcal{LC}((G,L),H)}\ \prod_{v\in V(G)} x_{f(v)}.
\end{equation}

We first note how equitable list $3$-colorings can be recovered from the coefficients of a suitable partition polynomial.
The strategy is to compute the polynomial $p_{(G,L)\to K_3}$ using~\Cref{thm:subexpo-P_t-groen-etal} and then extract from it exactly those coefficients corresponding to equitable color-class sizes.

\begin{lemma}\label{lem:coeff-counts}
Let $H=K_3$ and let
$p\coloneqq p_{(G,L)\to K_3}(x_1,x_2,x_3)$ be as in~\eqref{eq:partition-poly}.
For every triple $(a,b,c)\in\mathbb{Z}_{\ge 0}^3$ satisfying $a+b+c=n$, the coefficient of
$x_1^a x_2^b x_3^c$ in $p$ equals the number of list $3$-colorings $f$ of $(G,L)$ such that
$|f^{-1}(1)|=a$, $|f^{-1}(2)|=b$, and $|f^{-1}(3)|=c$.
\end{lemma}

\begin{proof}
Let $\mathcal{F}\coloneqq \mathcal{LC}((G,L),K_3)$.
By~\eqref{eq:partition-poly},
\[
p_{(G,L)\to K_3}(x_1,x_2,x_3)
=
\sum_{f\in\mathcal{F}} \prod_{v\in V(G)} x_{f(v)}.
\]
For each $f\in\mathcal{F}$ and each $i\in\{1,2,3\}$, define
\[
\alpha_i(f)\coloneqq \bigl|\{v\in V(G): f(v)=i\}\bigr|.
\]
Then
\[
\prod_{v\in V(G)} x_{f(v)}
=
x_1^{\alpha_1(f)}x_2^{\alpha_2(f)}x_3^{\alpha_3(f)}.
\]
For every triple $(a',b',c')\in\mathbb{Z}_{\ge 0}^3$ with $a'+b'+c'=n$, let
\[
\mathcal{F}_{a',b',c'}
\coloneqq 
\{\,f\in\mathcal{F}:\alpha_1(f)=a',\ \alpha_2(f)=b',\ \alpha_3(f)=c'\,\}.
\]
Then we have
\[
\mathcal{F}
=
\biguplus_{\substack{(a',b',c')\in\mathbb{Z}_{\ge 0}^3\\ a'+b'+c'=n}}
\mathcal{F}_{a',b',c'}.
\]
Using this partition, we obtain
\[
p_{(G,L)\to K_3}(x_1,x_2,x_3)
=
\sum_{\substack{(a',b',c')\in\mathbb{Z}_{\ge 0}^3\\ a'+b'+c'=n}}
|\mathcal{F}_{a',b',c'}|\;x_1^{a'}x_2^{b'}x_3^{c'}.
\]
Hence the coefficient of $x_1^a x_2^b x_3^c$ in $p_{(G,L)\to K_3}$ equals
$|\mathcal{F}_{a,b,c}|$, as claimed.
\end{proof}

\begin{lemma}\label{lem:eq-sum}
Let $n=3q+r$ with $r\in\{0,1,2\}$, and let
$p\coloneqq p_{(G,L)\to K_3}(x_1,x_2,x_3)$ be as in~\eqref{eq:partition-poly}.
Then the number of equitable proper list $3$-colorings of $(G,L)$, denoted by $\mathbf{\Lambda}(G,L)$, is
\[
\mathbf{\Lambda}(G,L)=
\begin{cases}
[x_1^q x_2^q x_3^q]\,p, & r=0,\\[2mm]
[x_1^{q+1}x_2^q x_3^q]\,p+[x_1^q x_2^{q+1}x_3^q]\,p+[x_1^q x_2^q x_3^{q+1}]\,p, & r=1,\\[2mm]
[x_1^{q+1}x_2^{q+1}x_3^q]\,p+[x_1^{q+1}x_2^q x_3^{q+1}]\,p+[x_1^q x_2^{q+1}x_3^{q+1}]\,p, & r=2.
\end{cases}
\]
\end{lemma}

\begin{proof}
Write $n=3q+r$ with $r\in\{0,1,2\}$.
If $(a,b,c)$ are the sizes of the three color classes in an equitable $3$-coloring, then the only possible
multisets $\{a,b,c\}$ are:
$\{q,q,q\}$ if $r=0$,
$\{q+1,q,q\}$ if $r=1$,
and $\{q+1,q+1,q\}$ if $r=2$.
By~\Cref{lem:coeff-counts}, for every ordered triple $(a,b,c)$ with $a+b+c=n$, the coefficient
$[x_1^a x_2^b x_3^c]\,p$ equals the number of list $3$-colorings of $(G,L)$ whose color-class sizes are exactly
$(a,b,c)$.
Therefore:
\begin{itemize}
    \item if $r=0$, the only feasible ordered triple is $(q,q,q)$;
    \item if $r=1$, the feasible ordered triples are
    $(q+1,q,q)$, $(q,q+1,q)$, and $(q,q,q+1)$;
    \item if $r=2$, the feasible ordered triples are
    $(q+1,q+1,q)$, $(q+1,q,q+1)$, and $(q,q+1,q+1)$.
\end{itemize}
Summing the corresponding coefficients yields exactly $\mathbf{\Lambda}(G,L)$.
\end{proof}
We need the following consequence of~\cite{groenland2019h}.

\begin{theorem}[Groenland, Okrasa, Rzążewski, Scott, Seymour, and Spirkl~\cite{groenland2019h}]\label{thm:subexpo-P_t-groen-etal}
Let $t\ge 4$ and let $H$ be a multigraph such that for all distinct $h,h'\in V(H)$,
$|N_H(h)\cap N_H(h')|\le 1$.
Then for every $P_t$-free graph $G$ on $n$ vertices and every list assignment
$L:V(G)\to 2^{V(H)}$, the polynomial $p_{(G,L)\to H}(x)$ can be computed in time
$2^{O(\sqrt{n\log n})}\cdot \mathrm{poly}(n)$.
\end{theorem}

We are now ready to prove~\Cref{thm:PT}.

\begin{proof}[Proof of Theorem~\ref{thm:PT}]
If $t=1$ or $t=2$, then $G$ is edgeless.
In this case every list $3$-coloring is simply a choice of an allowed color from $L(v)$ for each vertex $v$, and the
number of equitable list $3$-colorings can be computed in polynomial time by a dynamic program whose state records the three
current color-class sizes.
If $t=3$, then $G$ is a disjoint union of cliques.
A clique component of size at least $4$ has no $3$-coloring, so in that case the answer is $0$. Otherwise, each component has size at most $3$, and one can again use a polynomial-time dynamic programming algorithm over the components, where the state records the three current color-class sizes and the transitions respect the list assignment on the vertices of the next component.
Hence we may assume that $t\ge 4$.
We invoke~\Cref{thm:subexpo-P_t-groen-etal} with $H=K_3$.
Since for all distinct $h,h'\in V(K_3)$ we have
$|N_{K_3}(h)\cap N_{K_3}(h')|=1$,
the theorem applies, and we can compute the polynomial
$p\coloneqq p_{(G,L)\to K_3}(x_1,x_2,x_3)$
in time $2^{O(\sqrt{tn\log n})}\cdot \mathrm{poly}(n)$.
Write $n=3q+r$ with $r\in\{0,1,2\}$.
By~\Cref{lem:eq-sum}, the desired number of equitable list $3$-colorings is exactly the corresponding sum of
coefficients of $p$. Since every monomial of $p$ has total degree $n$, the polynomial $p$ contains at most $\binom{n+2}{2}=O(n^2)$
monomials. Moreover, each coefficient is at most $3^n$ and therefore has $O(n)$ bits. Hence, once $p$ is computed, extracting the required coefficients and summing them takes only $\mathrm{poly}(n)$ time.
Correctness follows from~\Cref{lem:eq-sum}.
\end{proof}

\bibliographystyle{abbrvurl}
\bibliography{ref}

\end{document}